\documentclass{article}
\usepackage[T1]{fontenc}
\usepackage{amssymb}
\usepackage{amsmath,amsthm}
\usepackage{float}
\usepackage{graphicx}
\usepackage{subcaption}
\usepackage{tikz-cd}

\newtheorem{theorem}{Theorem}[section]
\newtheorem{lemma}[theorem]{Lemma}

\newtheorem{corollary}[theorem]{Corollary}

\newtheorem{remark}[theorem]{Remark} 
\newtheorem{example}[theorem]{Example}

\begin{document}

\title{
\bf Formula for Dupin Cyclidic Cube and Miquel Point}

\author{
Jean Michel Menjanahary and
Rimvydas Krasauskas \\
{\small \em Institute of Computer Science, Vilnius University, Lithuania}
}

\date{\empty}
\maketitle

\begin{abstract}
Dupin cyclides are surfaces conformally equivalent to a torus, a circular cone, or a cylinder. Their patches admit rational bilinear quaternionic B\'ezier parametrizations and are used in geometric design and architecture.
Dupin cyclidic cubes are a natural trivariate generalization of Dupin cyclide patches. In this article, we derive explicit formulas for control points and weights of rational trilinear quaternionic B\'ezier parametrizations of Dupin cyclidic cubes and relate them with the classical construction of the Miquel point.
\end{abstract}

{\em Keywords:} {\small \em Dupin cyclidic cube, quaternionic-Bézier formula, Miquel point}

\def\R{\mathbb R}       
\def\RR{\mathbb R}       
\def\C{\mathbb C}       
\def\D{{\mathbb D}}
\def\H{\mathbb H}       
\def\HH{\mathbb H}       
\def\Sp{\mathbb{S}}              
\def\P{\mathbb{P}}       
\def\Ss{S}              
\def\S{\mathcal{S}} 
\def\Cr{\mathop{\mathrm{cr}}}

\def\re{\mathop{\mathrm{Re}}}
\def\im{\mathop{\mathrm{Im}}}
\def\rank{\mathop{\mathrm{rank}}}
\def\inv{\mathop{\mathrm{Inv}}}
\def\Qdr{{\cal Q}}
\newcommand{\sing}{\mathop{\mathrm{Sing}}}


\def\ic{\mathrm{i}} 
\def\ii{\mathbf{i}}
\def\jj{\mathbf{j}}
\def\kk{\mathbf{k}}

\def\B{B\'ezier }
\def\M{M\"obius }

\section{Introduction}

Dupin cyclides, i.e., surfaces conformally equivalent to a torus, circular cone, or cylinder, have versatile applications in geometric design and architecture. They have circular curvature lines. Their patches bounded by 4 circles, which are curvature lines, allow rational bilinear quaternionic B\'ezier (QB) parametrizations and offer significant advantages in modeling complex shapes.
Building on the foundational concepts of Dupin cyclide principal patches, Dupin cyclidic (DC) cubes represent a natural trivariate extension. We can effectively model these higher-dimensional structures by employing rational trilinear QB parametrizations. 

This paper presents explicit formulas for the control points and weights necessary for the rational trilinear QB parametrizations of DC cubes. Additionally, we establish a connection between these parametrizations and the classical Miquel point construction. The quaternionic representations of DC cubes were recently used in \cite{menjanahary2024cl}, and the present paper supports these results. Using geometric constructive derivation, a preliminary QB formula for DC cubes was derived in \cite{zube2013quat}.

\section{Quaternions and Inversions}\label{s:1}
The algebra of quaternions $\HH$ is the real non-commutative algebra generated by $\{\ii, \jj,\kk\}$ satisfying the product rules $\ii^2\!=\!\jj^2\!=\!\kk^2\!=\!\ii\jj\kk\!=\!-1$. It is a 4-dimensional real vector space with the standard basis 
$\{1,\ii,\jj,\kk\}$.
For a quaternion $q$ written in the algebraic form 
$q = r + x \ii + y \jj + z \kk$, 
define the real part $\re(q)=r$, the imaginary part 
$\im(q)= q-\re(q)$,
the conjugate $\Bar{q}=\re(q)-\im(q)$, and the norm $|q|=\sqrt{q\Bar{q}}$. 
The algebra $\HH$ is also a division ring, meaning that every non-zero element is invertible. If $q\neq0$, its inverse is $q^{-1}=\Bar{q}/|q|^2$.
The properties of conjugation and norm are the same as those of complex numbers, but care must be taken to account for non-commutativity when permuting product elements; e.g., $\overline{qp}=\bar{p}\bar{q}$. We refer to \cite{zube2013quat, zube2015rep} for further details about quaternionic products and their relation to the standard dot and cross products in $\RR^3$. The Euclidean space $\RR^3$ here is identified with the space of imaginary quaternions $\im\!\HH=\{q\in \HH\mid \re(q)=0\}$.

Since we are dealing with objects composed of circles and lines in $\RR^3$, it is natural to use inversion transformations that preserve the set of circles and lines and angles between crossing curves, known as conformal transformations.
In the quaternionic framework, an inversion $\inv_q^r$ with respect to a sphere of center $q\in \im\!\HH$ and radius $r>0$ can be written explicitly as 
\[\text{Inv}_q^{r}(p)=q-r^2(p-q)^{-1}\in \im\!\HH\]
for all $p\in \im\!\HH$.
On the compactified space $\widehat{\RR}^3=\RR^3\cup \{\infty\}$, which is identified with $\im\widehat{\HH}=\im\!\HH \cup \{\infty\}$, $\inv_q^r$ is an involution transformation mapping the center $q$ to $\infty$ and vice versa.
The group generated by inversions is called the group of \M transformations. Euclidean similarities are particular cases of \M transformations; see \cite[Sec.~2]{zube2015rep} for more details.

\section{Formula for Circles and Dupin Cyclides}
Define the quaternionic fraction 
\[
\frac{U}{W} =
\begin{cases}
  U W^{-1} & \text{if } W \neq 0, \\
  \infty & \text{if } W = 0.
\end{cases}
\]
A rational quaternionic B\'ezier (QB) curve $C(t)$ of degree $n$ is defined by the following data:
\begin{itemize}
\item Control points $p_i\in \HH$, $i=0,\ldots, n$;
\item Weights $w_i\in \HH$, $i=0,\ldots, n$;
\end{itemize}
such that 
\[
C(t) = \frac{\sum_{i=0}^{n}p_iw_iB_i^n(t)}{\sum_{i=0}^{n}w_iB_i^n(t)}, \quad t\in [0,1],
\]
where $B_i^n(t) ={n \choose i}(1-t)^{n-i}t^i$ are Bernstein basis polynomials. 
The matrix
\begin{align*}
\begin{pmatrix}
u_i\\
w_i
\end{pmatrix}_{i=0\dots n}
=
\begin{pmatrix}
p_iw_i\\
w_i
\end{pmatrix}_{i=0\dots n}
\end{align*}
is called the homogeneous representation of the curve $C(t)$.
We will be interested in the linear case ($n=1$). 

\begin{remark}
\label{rem:inv-trans}
QB formulas are preserved by inversions in the following sense: an inversion $\inv_q^r$ maps a QB formula with homogeneous control points $(u_i,w_i)$ to a QB formula with homogeneous control points $(u_i',w_i')$ such that
 \begin{equation}
u_i' =qu_i-(r^2+q^2)w_i, \quad w_i'=u_i-qw_i.
 \end{equation}
\end{remark}

To model curves in $\widehat{\RR}^3$, let's standardize the condition for an arbitrary pair $(U,W)$ of quaternions to define a point $UW^{-1}\in \widehat{\RR}^3$.
By identifying $\HH^2$ with $\RR^8$, define the quadratic form $\Ss$ in $\RR^8$ by
\begin{equation}\label{Study}
\Ss(u,w) = \frac{u \bar w + w \bar u}{2}, \quad (u,w)\in \HH^2.    
\end{equation}
The quadric in $\RR P^7$ (real projectivization  $\HH^2$)  defined by $\Ss(u,w) = 0$ is called the \emph{Study quadric}, which we denote by $\Ss$ as well. Let $\pi: \HH^2\rightarrow \HH\cup \{\infty\}$ such that $\pi(u,w)=\frac{u}{w}$. The following lemma is straightforward:

\begin{lemma}
$\pi(u,w)\in \widehat{\RR}^3$ if and only if $(u,w)\in S$.
\end{lemma}

\noindent
To design a QB curve in $\widehat{\RR}^3$, we follow the following routines:
\begin{itemize}
\item The control points $p_i$ are contained in $\widehat{\RR}^3$;
\item The homogeneous control points $(p_iw_i,w_i)$ are contained in the Study quadric;
\item Consider the pair $(U(t), W(t))$ as a standard B\'ezier curve (with control points $(p_iw_i,w_i)$) with real weights in the Study quadric. Then, apply the projection $\pi$ to get a curve in $\widehat{\RR}^3$; see the diagram below.
\end{itemize}

\[
\begin{tikzcd}
{[0,1]}\arrow[r, "\iota"] \arrow[dr, "\pi \circ \iota"] & {\Ss\subset \RR\mathbb P^7}  \arrow[d, "\pi"'] &
{t} \arrow[r, maps to,"\iota"] \arrow[dr, maps to,"\pi \circ \iota"] & {\left(U(t), W(t)\right)}  \arrow[d, maps to,"\pi"']\\
& \RR^3\cup \{\infty\} &
& \frac{U(t)}{W(t)}
\end{tikzcd}
\]

\begin{example}\rm 
 A circular arc with endpoints $p_0$, $p_1$ and a tangent vector $v_1$ at $p_0$ can be parametrized using the QB formula
 \[C(t) = \frac{\sum_{i=0}^{1}u_iB_i^1(t)}{\sum_{i=0}^{1}w_iB_i^1(t)},\]
 where
 \begin{equation}
  \begin{pmatrix}
  u_0 & u_1\\
  w_0 & w_1
 \end{pmatrix}
 =
 \begin{pmatrix}
  p_0 & p_1(p_1-p_0)^{-1}v_1\\
  1 & (p_1-p_0)^{-1}v_1
 \end{pmatrix}.
 \end{equation}
 Such formulas can be found in \cite{menjanahary2024cl,zube2013quat,zube2015rep}. 
 Note that if $p_0=\infty$, then we have a semi-line starting from $p_1$ in the direction of $v_1$. This semi-line can be parametrized using the control points:
 \begin{equation}
  \begin{pmatrix}
  u_0 & u_1\\
  w_0 & w_1
 \end{pmatrix}
 =
 \begin{pmatrix}
  1 & -p_1v_1\\
  0 & -v_1
 \end{pmatrix}.
 \end{equation}
\end{example}

\begin{remark}\rm 
A reparametrization of the arc is obtained if we multiply the weight $w_1$ by a constant $\lambda>0$. If $\lambda<0$, then a parametrization of the complementary arc is obtained. 
 If the arc is defined by the endpoints $p_0$, $p_1$ and a point $q$ on the complementary arc, then the weights can be assigned as
 \[w_0 = (q - p_0)^{-1}, \qquad w_1 = (p_1 - q)^{-1}.\]
\end{remark}

A bivariate generalization of the QB formula of circular arcs yields 
a so-called Dupin cyclide principal patch or simply a principal patch.
They are quad patches bounded by 4 circular arcs intersecting orthogonally at the corner points.
Note that the 4 corner points are always cocircular; see Figure~\ref{fig:principal-p}. This circularity condition can be interpreted in terms of cross-ratio. 
The cross-ratio between points $p_0,p_1,p_2,p_3\in \im\!\HH$ is defined as
\[\Cr(p_0,p_1,p_2,p_3)=(p_0-p_1)(p_1-p_2)^{-1}(p_2-p_3)(p_3-p_0)^{-1},\]
whenever the product is well-defined.
\begin{remark}\rm 
\label{rem:cross-ratio}
Four points $p_0,p_1,p_2,p_3\in \im\!\HH$ are cocircular if and only if their cross-ratio is real; see \cite[Lemma~2.3]{zube2015rep}.  
\end{remark}

A principal patch is uniquely determined by its four cocircular corner points and tangent vectors $v_1$, $v_2$ at one corner point. The tangent vectors at other points are obtained by reflection along the respective edges.

To simplify the notation in a bilinear QB formula, we use the binary indices in reversed order $0=00$, $1=10$, $2=01$, $3=11$.

\begin{lemma}
Let a principal patch be defined by cocircular corner points $p_0$, $p_1$, $p_2$, $p_3$ and orthogonal tangent vectors $v_1$ and $v_2$ 
at $p_0$ and let $v_3=v_1v_2$. Then, this patch can parametrized using the bilinear QB formula 
 \[P(s,t)=\frac{\sum_{i,j=0}^{1}p_{ij}w_{ij}B^1_{i}(s)B^1_j(t)}{\sum_{i,j=0}^{1}w_{ij}B^1_{i}(s)B^1_j(t)},\]
with the following homogeneous representation:
\begin{itemize}
\item[(i)]
    If $p_0 = \infty$ and 
    $p_1 \ne p_2$, then the first control point is $(u_0,w_0)=(1,0)$, and the others are $(p_iw_i,w_i)$  such that 
    \begin{equation}\label{inf-weights}
    w_1=-v_1,\ w_2=-v_2,\ w_3=(p_1-p_2) v_3.
    \end{equation}
\item[(ii)]
If all control points are finite, only $p_1$ and $p_2$ may coincide with $p_3$, then
\begin{align}\label{fin-weights}
w_0 = 1, \ w_1=q_{1} v_1,\ w_2=q_{2}v_2, \ 
w_3 = q_{3}\left(q_{1}-q_{2}\right)v_3,
\end{align}
where $q_{i}=(p_i-p_0)^{-1}$ for $i=1,2,3$.
\end{itemize}
\end{lemma}
\begin{proof}
This is proved in \cite[Theorem 2.2]{menjanahary2024cl}.
\end{proof}

\begin{figure}
    \centering
    \includegraphics[width=0.45\linewidth]{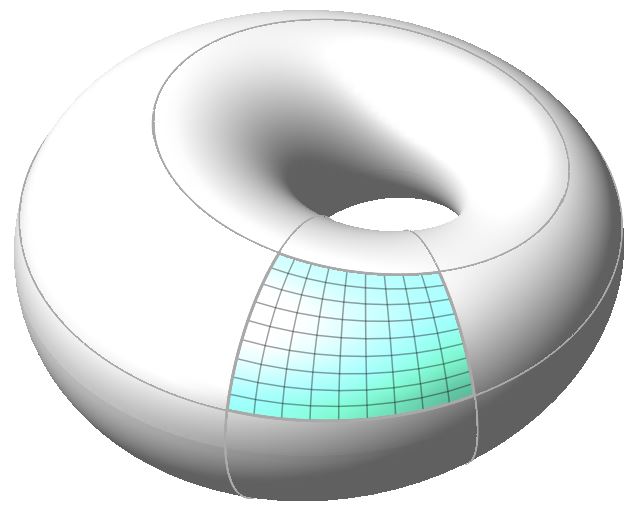}\hskip1cm
    \includegraphics[width=0.45\linewidth]{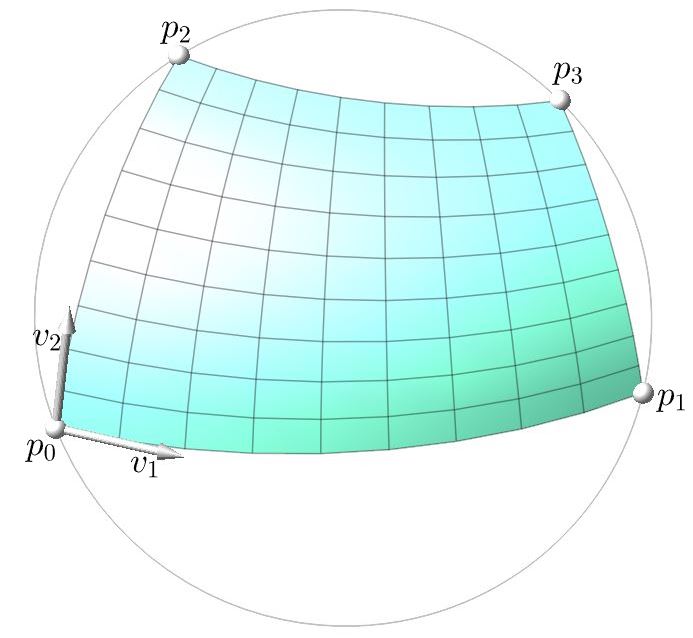}
    \caption{A Dupin cyclide principal patch.}
    \label{fig:principal-p}
\end{figure}

\section{Formulas for Dupin cyclidic Cubes}

A DC cube is a trilinear rational quaternionic map
\[
F: [0,1]^3 \to \widehat{\RR}^3, \quad F = \frac{U}{W}, \ U, W \in \HH[s,t,u],
\]
such that all three partial derivatives $\partial_s F$, 
$\partial_t F$, $\partial_u F$ are mutually orthogonal, and the Jacobian $\mathrm{Jac}(F)$ is not identically zero.
As we will investigate, DC cubes can be expressed using QB formulas with their 8 corner points as control points and quaternionic weights.
We highlight the following essential properties, the details can be found in \cite{menjanahary2024cl}:
\begin{itemize}
\item The 8 corner points of a DC cube lie on a sphere or a plane.
\item A DC cube is uniquely determined by its 3 adjacent faces at a corner point, see Fig \ref{fig:DC-finite}(b).
\item The last control point $p_7$ can be derived using Miquel theorem about the intersection of 3 circles; see Figure~\ref{fig:DC-inf}(a) and Figure~\ref{fig:DC-finite}(c).
\end{itemize}

\begin{figure}[H]
    \centering
    \begin{subfigure}[b]{0.5\textwidth}
        \centering
         \includegraphics[width=5cm]{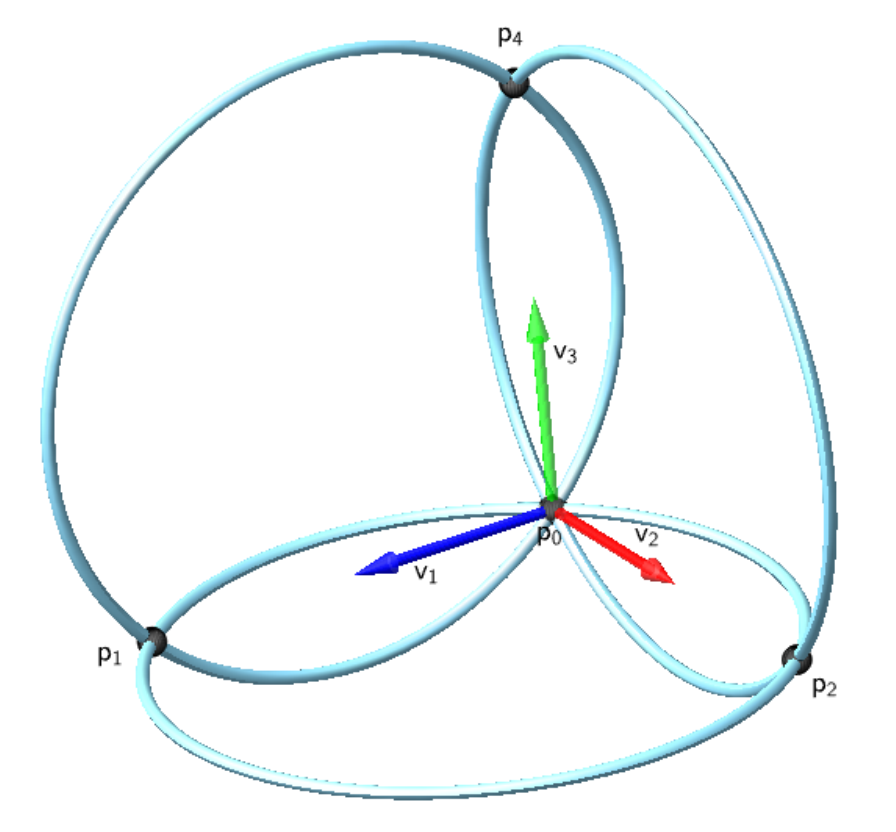}
        \caption{Initial data.}
    \end{subfigure}
    \hfill
    \begin{subfigure}[b]{0.45\textwidth}
        \centering
        \includegraphics[width=5cm]{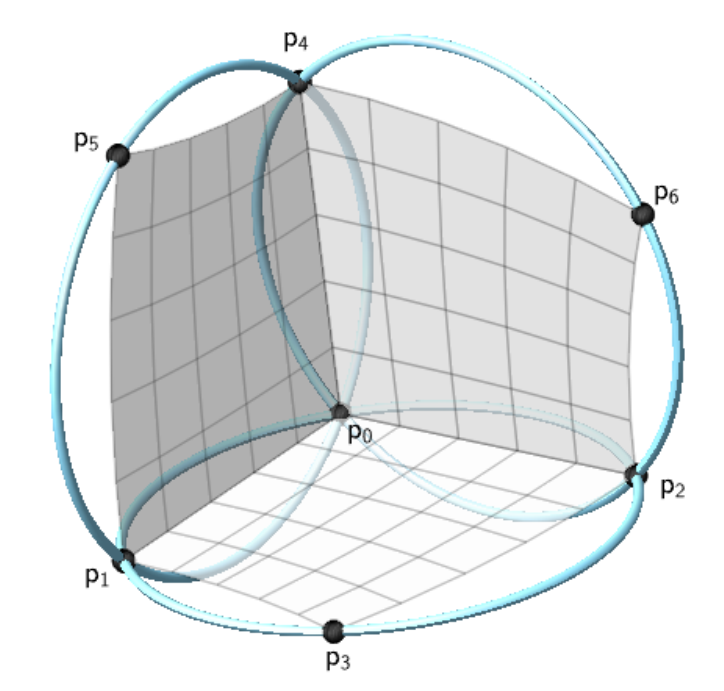}
        \caption{Three compatible faces.}
    \end{subfigure}\\
    \hfill
    \begin{subfigure}[b]{0.5\textwidth}
        \centering
        \includegraphics[width=5cm]{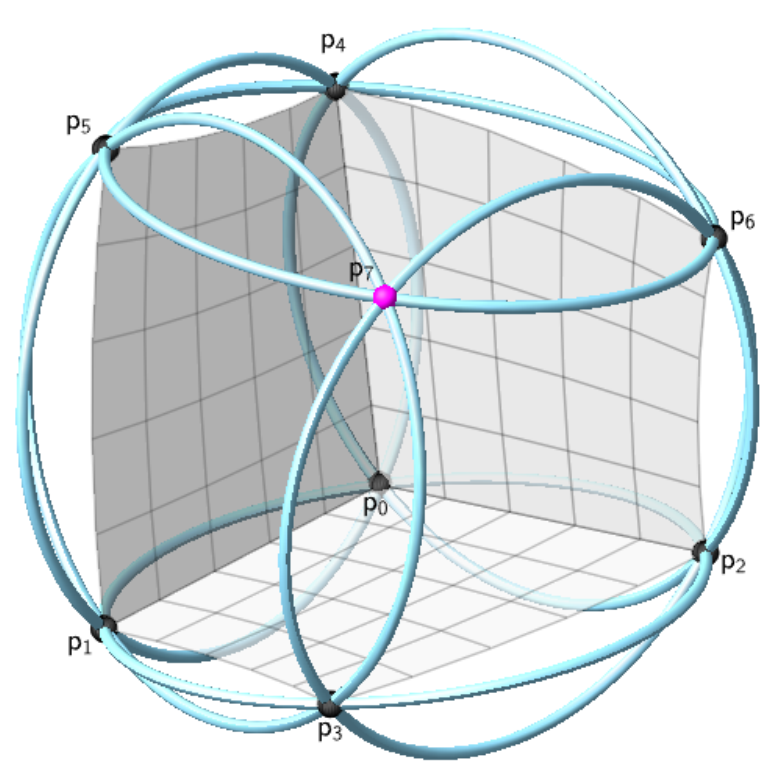}
        \caption{The Miquel point $p_7$.}
    \end{subfigure}
    \hfill
    \begin{subfigure}[b]{0.45\textwidth}
        \centering
        \includegraphics[width=5cm]{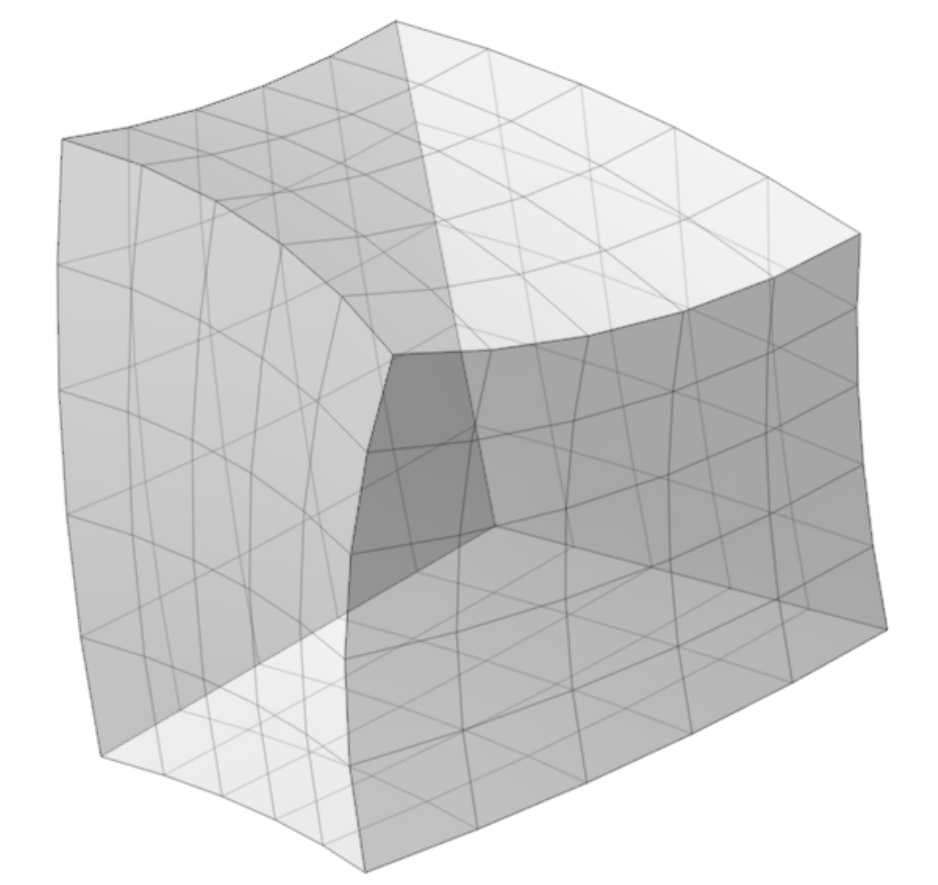}
        \caption{The resulting DC cube.}
    \end{subfigure}
    \caption{Steps of the construction of a DC cube.}
    \label{fig:DC-finite}
\end{figure}

\noindent
To construct compatible 3 adjacent faces at a corner point, say $p_0$, we use an orthogonal frame $(v_1,v_2,v_3)$ at $p_0$ and three (arbitrary) corner points $p_1$, $p_2$, $p_4$; see Figure~\ref{fig:DC-finite}(a).
The 3 compatible faces are determined if we fix the fourth corner point of each patch:  $p_{i+j}$ on the circle ${\cal C}(p_0,p_i,p_j)$, $ij\in \{12,14, 24\}$. Here ${\cal C}(p_0,p_i,p_j)$ denotes the circle through the 3 points $p_0,p_i,p_j$.
The formula for the last control point $p_7$ depending on the other 7 corner points will be derived in Corollaries~\ref{cor:miquel} and \ref{cor:miquel-gen}.

This paper's derivation first constructs a DC cube using $p_0=\infty$ and then applies inversions to derive the QB formula for general DC cubes. Similar to the bilinear case, to simply express a QB formula for DC cubes, we use the binary indices in reversed order $0=000$, $1=100$, $2=010$, $3=110$, $4=001$, $5=101$, $6=011$, $7=111$.

\begin{theorem}
\label{thm:p0-inf}
Let a DC cube be defined by 8 corner points $p_0=\infty$, $p_1,p_2,p_4\in \im\!\HH$, $p_3,p_5,p_6$ on the lines ${\cal L}(p_1,p_2), {\cal L}(p_1,p_4), {\cal L}(p_2,p_4)$ respectively, with the associated Miquel point $p_7$ and an orthonormal frame $(v_1,v_2,v_3=v_1v_2)$ at $p_0$. Then, we can parametrize this DC cube using the trivariate QB formula
\[F(s,t,u)=\frac{\sum_{i,j,k=0}^{1}p_{ijk}w_{ijk}B^1_i(s)B^1_j(t)B^1_k(u)}{\sum_{i,j,k=0}^{1}w_{ijk}B^1_i(s)B^1_j(t)B^1_k(u)},\]
with the first 7 homogeneous control points: 
\begin{align*}
\begin{pmatrix}
1 & -p_1v_1 & -p_2v_2 & p_3(p_1-p_2)v_3 & -p_4v_3 & p_5(p_4-p_1)v_2 & p_6(p_2-p_4)v_1\\
0 & -v_1 & -v_2 & (p_1-p_2)v_3 & -v_3 & (p_4-p_1)v_2 & (p_2-p_4)v_1
\end{pmatrix},
\end{align*}
and with the last control point $(p_7w_7,w_7)$, where $w_7$ has the following equivalent expressions
\begin{align}
w_7 &= (p_7-p_1)^{-1}(p_4-p_1)(p_3-p_5)(p_1-p_2)\label{eq:w7_1}\\
    &= (p_7-p_2)^{-1}(p_1-p_2)(p_6-p_3)(p_2-p_4)\label{eq:w7_2}\\
    &= (p_7-p_4)^{-1}(p_2-p_4)(p_5-p_6)(p_4-p_1).\label{eq:w7_3}
\end{align}
\end{theorem}
\begin{proof}
Let $f_{0123}$, $f_{0415}$ and $f_{0246}$ be the initial three faces of the DC cube meeting at $p_0=\infty$; see Figure \ref{fig:DC-inf}(b). Using Formula \eqref{inf-weights}, we obtain the presented formula for $w_i$, $i=0,\ldots,6$. Note that the frames at $p_1$, $p_2$, and $p_4$ for the DC cube are the same.
On the face $f_{4567}$, we compute the weights using Formula \eqref{fin-weights}. This gives
\begin{align*}
w_4'&=1, \quad
w_5'=(p_5-p_4)^{-1}v_1,\quad
w_6'= (p_6-p_4)^{-1}v_2,\\
w_7'&=(p_7-p_4)^{-1}(p_6-p_4)^{-1}(p_6-p_5)(p_5-p_4)^{-1}v_3.
\end{align*}
To get the compatibility at $p_4$, we need to multiply such weights with $-v_3$. This gives
\begin{align*}
w_4''&=-v_3,\quad
w_5''=(p_5-p_4)^{-1}v_2,\quad
w_6''=-(p_6-p_4)^{-1}v_1,\\
w_7''&=(p_7-p_4)^{-1}(p_6-p_4)^{-1}(p_6-p_5)(p_5-p_4)^{-1}.
\end{align*}
To get the compatibility at $p_5$ and $p_6$, we multiply $w_5''$ by $\lambda_1=(p_4-p_1)(p_5-p_4)$ and $w_6''$ by $\lambda_2=-(p_2-p_4)(p_6-p_4)$. Note  that $\lambda_1$ and $\lambda_2$ are real because the points $p_1,p_4,p_5$ and similarly $p_2,p_4,p_6$ are collinear. Hence, a reparametrization of the face $f_{4567}$ using $w_4''=w_4$, $\lambda_1 w_5''=w_5$, $\lambda_2 w_6''=w_6$ and $\lambda_1\lambda_2w_7''=w_7$, which is the compatible weight at $p_7$. In the product $\lambda_1\lambda_2w_7''$, the factors $p_5-p_4$ and $p_6-p_4$ of $\lambda_1$ and $\lambda_2$ will be eliminated, giving the formula \eqref{eq:w7_3} for $w_7$.
By studying the compatibility similarly on the faces $f_{1357}$ and $f_{2637}$, we obtain alternative formulas for $w_7$ in \eqref{eq:w7_1} and \eqref{eq:w7_2}. It follows from the compatibility lemma \cite[Lemma 3.3]{menjanahary2024cl} that the 3 found weights have to coincide, giving a compatible parametrization of the DC cube; see Figure~\ref{fig:DC-inf}(c).
\end{proof}

\begin{figure}[H]
    \centering
    \begin{subfigure}[b]{\textwidth}
        \centering
         \includegraphics[width=6cm]{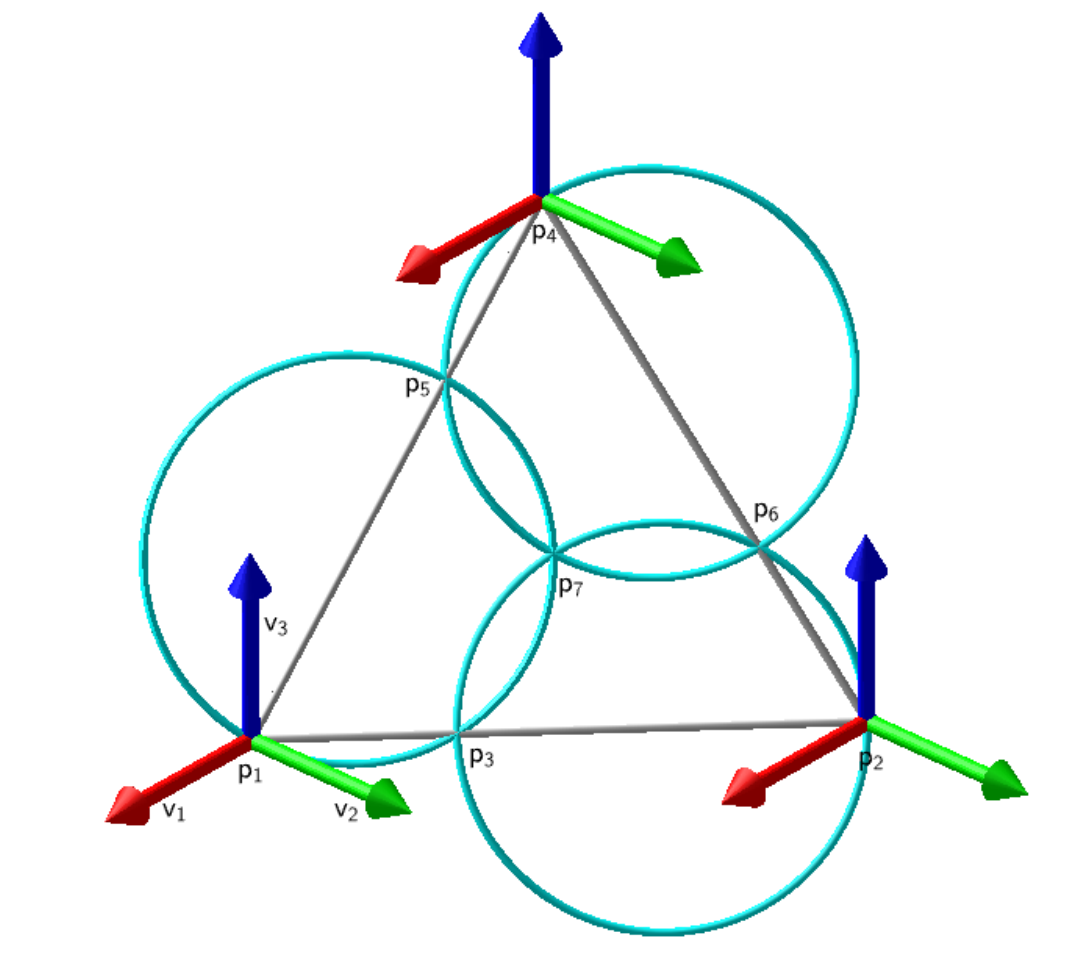}
        \caption{Initial data for a DC cube with one corner point on infinity.}
    \end{subfigure}\\
    \hfill
    \begin{subfigure}[b]{0.5\textwidth}
        \centering
        \includegraphics[width=5cm]{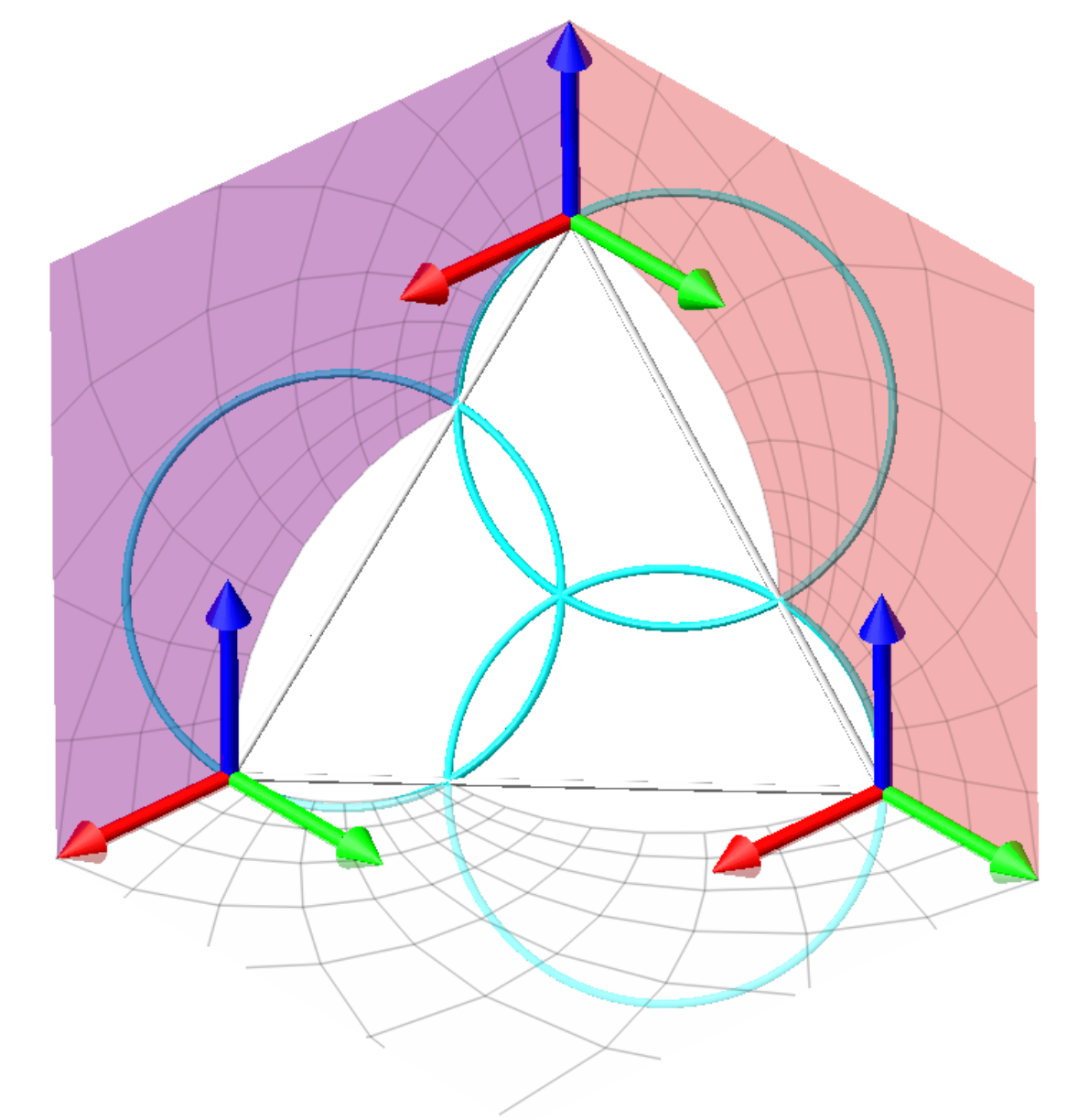}
        \caption{Three compatible faces of a DC cube with a common intersection point on infinity.}
    \end{subfigure}
    \hfill
    \begin{subfigure}[b]{0.4\textwidth}
        \centering
        \includegraphics[width=5cm]{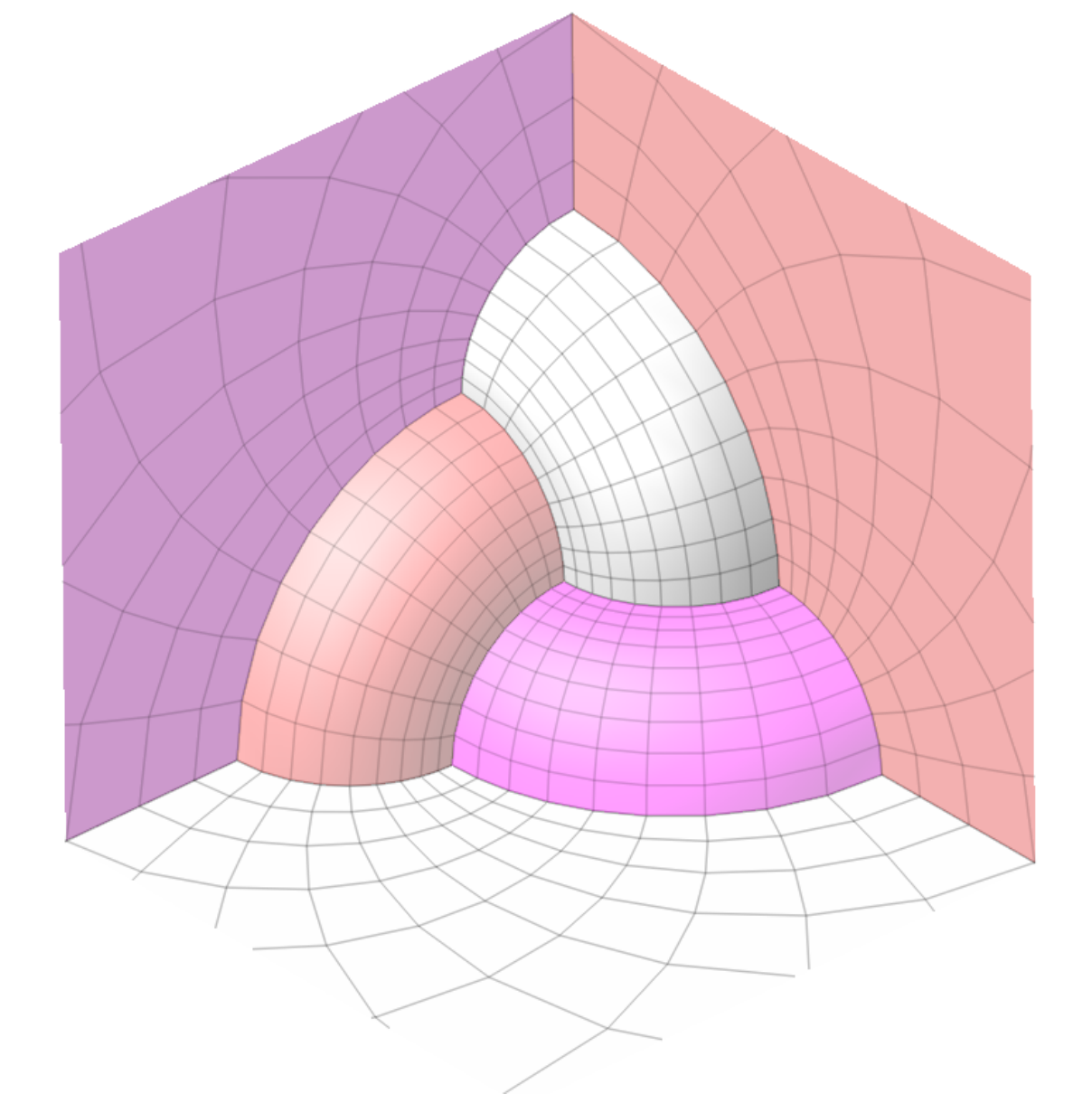}
        \caption{The resulting 6 faces of the DC cube from 3 compatible faces.}
    \end{subfigure}
    \caption{DC cube construction steps with one control point on infinity.}
    \label{fig:DC-inf}
\end{figure}

\begin{corollary}
\label{cor:miquel}
The Miquel point $p_7$ can be expressed as 
\begin{align}
p_7 &= p_1+ A(A-B)^{-1} (p_2-p_1)\label{eq:miq1}\\
    &= p_2+ B(B-C)^{-1} (p_4-p_2)\label{eq:miq2}\\
    &= p_4+ C(C-A)^{-1} (p_1-p_4)\label{eq:miq3},
\end{align}
where $A$, $B$, $C$ are the right-quaternionic factors of $w_7$, namely
\begin{align*}
A&=(p_4-p_1)(p_3-p_5)(p_1-p_2),\\
B&=(p_1-p_2)(p_6-p_3)(p_2-p_4),\\
C&=(p_2-p_4)(p_5-p_6)(p_4-p_1).
\end{align*}
\end{corollary}
\begin{proof}
From \eqref{eq:w7_1} and \eqref{eq:w7_2}, we have $w_7=(p_7-p_1)^{-1}A=(p_7-p_2)^{-1}B$. This implies
\begin{align*}
 BA^{-1}
 &=(p_7-p_2)(p_7-p_1)^{-1} \\
 &=(p_7-p_1+p_1-p_2)(p_7-p_1)^{-1}\\
 &=1+(p_1-p_2)(p_7-p_1)^{-1}.
\end{align*}
Hence $(p_7-p_1)^{-1}=(p_1-p_2)^{-1}(BA^{-1}-1)=(p_1-p_2)^{-1}(B-A)A^{-1}$, i.e, $p_7-p_1=A(B-A)^{-1}(p_1-p_2)=A(A-B)^{-1}(p_2-p_1)$. We obtain \eqref{eq:miq1} by adding $p_1$ on both sides. The expressions \eqref{eq:miq2} and \eqref{eq:miq3} can be obtained similarly by considering other pairs of expressions for $w_7$.
\end{proof}

We apply inversions to relate the formula in Theorem \ref{thm:p0-inf} to a general formula for DC cubes with finite control points.

\begin{theorem}
\label{thm:p0-finite}
Let a DC cube be defined by 8 corner points $p_0$, $p_1$, $p_2$, $p_4\in \im\!\HH$, $p_3$ on the circle ${\cal C}(p_0,p_1,p_2)$, $p_5$ on the circle ${\cal C}(p_0,p_1,p_4)$, $p_6$ on the circle ${\cal C}(p_0,p_2,p_4)$, the associated Miquel point $p_7$, and an orthonormal frame $(v_1,v_2,v_3=v_1v_2)$ at $p_0$. Let $q_{i}=(p_i-p_0)^{-1}$ for $i=1,\dots,7$. Then, this cube can be parametrized by the QB formula
\[F(s,t,u)=\frac{\sum_{i,j,k=0}^{1}p_{ijk}w_{ijk}B^1_i(s)B^1_j(t)B^1_k(u)}{\sum_{i,j,k=0}^{1}w_{ijk}B^1_i(s)B^1_j(t)B^1_k(u)},\]
where
\begin{align*}
w_0 &= 1, \qquad w_1=q_{1} v_1,\qquad  \; w_2=q_{2}v_2, \qquad  w_4=q_{4}v_3\\ 
w_3 &= q_{3}\left(q_{1}-q_{2}\right)v_3,
\;\; w_5 = q_{5}\left(q_{4}-q_{1}\right)v_2,
\;\; w_6 = q_{6}\left(q_{2}-q_{4}\right)v_1,\\
w_7 &= -q_{7}(q_{7}-q_{1})^{-1}(q_{4}-q_{1})(q_{3}-q_{5})(q_{1}-q_{2}),\\
    &= -q_{7}(q_{7}-q_{2})^{-1}(q_{1}-q_{2})(q_{6}-q_{3})(q_{2}-q_{4})\\
    &= -q_{7}(q_{7}-q_{4})^{-1}(q_{2}-q_{4})(q_{5}-q_{6})(q_{4}-q_{1}).
\end{align*}
\end{theorem}
\begin{proof}
 This is equivalent to the formula in Theorem \ref{thm:p0-inf} using inversions as addressed in Remark \ref{rem:inv-trans}. For instance, let us consider the derivation of $w_7$. We apply first $\inv_{p_0}^1$ and all the control points are transformed to $p_0'=\infty$ and $p_i'=p_0-q_{i}$, $i=1,\ldots,7$. 
 By Theorem~\ref{thm:p0-inf}, Formula~\eqref{eq:w7_1}, we have $w_7'=(q_{7}-q_{1})^{-1}(q_{4}-q_{1})(q_{3}-q_{5})(q_{1}-q_{2})$. Hence, by applying the same inversion, we obtain
 $w_7=(p_7'-p_0)w_7'=-q_{7}w_7'$. This coincides with the first displayed formula for $w_7$. The other two equivalent formulas follow from Formulas~\eqref{eq:w7_2} and \eqref{eq:w7_3}. 
\end{proof}

The following result follows from Corollary \ref{cor:miquel} by applying inversions. 

\begin{corollary}
\label{cor:miquel-gen}
With the notations in Theorem $\ref{thm:p0-finite}$, the 8th control point $p_7$ of the DC cube, analogue of the Miquel point on the plane, can be expressed as 
\begin{align}
p_7 &= p_0 + [q_{1}+ A'(A'-B')^{-1} (q_{2}-q_{1})]^{-1},
\end{align}
where
\begin{align*}
A'&=(q_{4}-q_{1})(q_{3}-q_{5})(q_{1}-q_{2}),\\
B'&=(q_{1}-q_{2})(q_{6}-q_{3})(q_{2}-q_{4}).
\end{align*}
\end{corollary}

\begin{figure}
    \centering
    \includegraphics[width=0.42\linewidth]{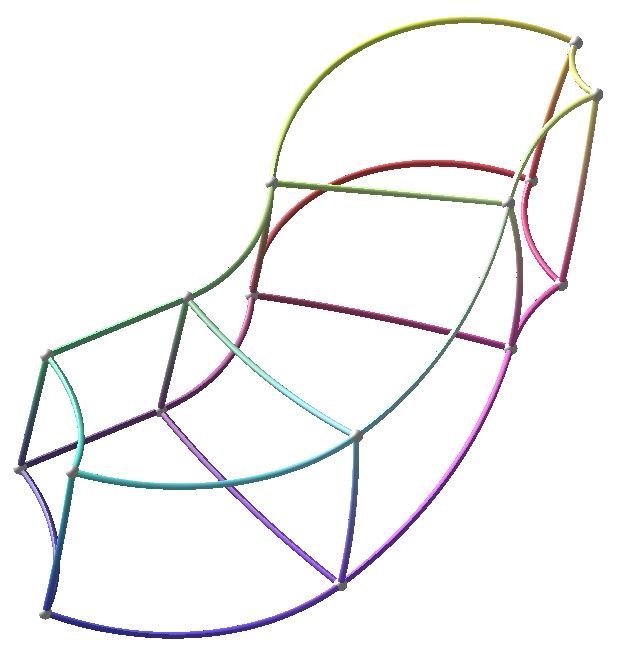}\hskip1cm
    \includegraphics[width=0.45\linewidth]{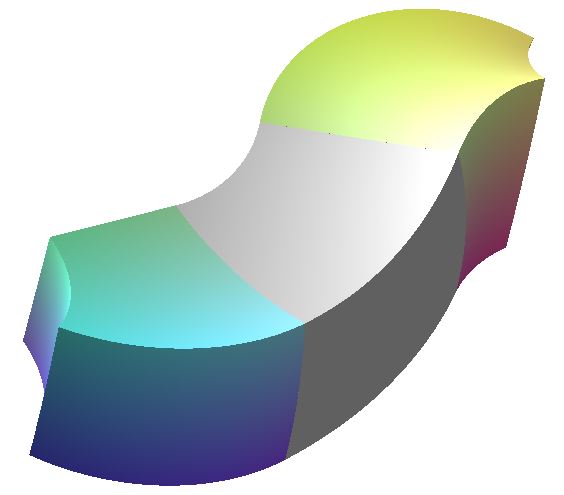}
    \caption{A 3D-cyclidic net.}
    \label{fig:3d-net}
\end{figure}

\begin{remark}\rm
The QB formula presented in this paper can be used to model 3D cyclidic nets \cite{bobenko2008discrete}, which are volumetric objects made up of smoothly blended DC cubes; see Figure~\ref{fig:3d-net}. A 3D cyclidic net is defined by its vertices, which are (prescribed) corner points of the DC cubes, and a frame at one vertex. This frame can be extended to other corner points by reflecting it with respect to the median planes of two neighboring vertices, hence fully determining the DC cubes. Each cube can be parametrized using Theorem~\ref{thm:p0-finite}.
\end{remark}

\section*{Acknowledgements}

This work is part of a project that has received funding from the European Union’s Horizon 2020 
research and innovation programme under the Marie Skłodowska-Curie grant agreement No 860843.

\begingroup
\small 
\bibliographystyle{plain}
\bibliography{references}

\begin{thebibliography}{1}

\bibitem{bobenko2008discrete}
A.I. Bobenko and Y.B. Suris.
\newblock {\em Discrete differential geometry: integrable structure}.
\newblock AMS, 2008.

\bibitem{menjanahary2024cl}
J.M. Menjanahary, E.~Hoxhaj, and R.~Krasauskas.
\newblock Classification of dupin cyclidic cubes by their singularities.
\newblock {\em Computer Aided Geometric Design}, 112:102362, 2024.

\bibitem{zube2013quat}
S.~Zube.
\newblock Quaternionic {B}{\'e}zier curves, surfaces and volume.
\newblock {\em Lietuvos matematikos rinkinys}, 54:79--84, 2013.

\bibitem{zube2015rep}
S.~Zube and R.~Krasauskas.
\newblock Representation of {D}upin cyclides using quaternions.
\newblock {\em Graphical Models}, 82:110--122, 2015.

\end{thebibliography}
\endgroup
\end{document}